\documentclass{article}

\usepackage{import}

\usepackage{amsmath,amsthm,mathtools}
\usepackage{amssymb}
\usepackage{wasysym}
\usepackage[dvipsnames]{xcolor}
\colorlet{mymathcolor}{blue}
\colorlet{mymathcoloremph}{red}
\usepackage{xpatch,letltxmacro}
\usepackage[no-math]{fontspec}
\usepackage{unicode-math}
\unimathsetup{
    mathbf = sym,
    warnings-off = {
        mathtools-colon,
        mathtools-overbracket
    }
}
\newfontfamily\upit[Path={./}]{lmromanunsl10-regular.otf}

\newfontfamily\dhstyle[Path={./},Scale=MatchLowercase,UprightFont={*-regular.otf},SlantedFont={*-oblique.otf}]{lmromandunh10}

\newfontfamily{\dcstyle}[%
  ItalicFont={lmsansdemicond10-oblique.otf},%
  SlantedFont={lmsansdemicond10-oblique.otf},%
]{lmsansdemicond10-regular.otf}

\setmathfontface\symupit[%
  Path={./}]{lmromanunsl10-regular.otf}
\setmathfontface\symitup[Path={./}]{lmromanunsl10-regular.otf}
\xpatchcmd{\qopname}{#3}{\textcolor{mymathcolor}{#3}}{}{}

\LetLtxMacro\latexsqrt\sqrt
\RenewDocumentCommand{\sqrt}{om}{%
  \colorlet{current}{.}
  \IfNoValueTF{#1}
    {\textcolor{mymathcolor}{\latexsqrt{\textcolor{current}{#2}}}}%
    {\textcolor{mymathcolor}{\latexsqrt[#1]{\textcolor{current}{#2}}}}%
}

\newcommand\ovl[1]{\colorlet{temp}{.}\color{mymathcolor}\overline{\color{temp}#1}\color{temp}}
\newcommand\mystar{\colorlet{temp}{.}\color{mymathcolor}\star\color{temp}}

\newcommand\mymod{\colorlet{temp}{.}\color{mymathcolor}\mod\color{temp}}
\newcommand\mybmod{\colorlet{temp}{.}\color{mymathcolor}\bmod\color{temp}}

\def\hat{\mathaccent "705E\relax}

\makeatletter
\DeclareRobustCommand{\cgenfrac}[5]{%
  \def\@tempa{#2#3}%
  \edef\@tempb{%
    \@nx\@cgenfrac{#1}%
    \@mathstyle{#5}%
    \csname
      @@%
      \ifx @#4@over\else above\fi
      \ifx\@tempa\@empty\else withdelims\fi
    \endcsname
  }%
  \@tempb{#2#3#4}%
}
\newcommand\@cgenfrac[6]{%
  {#2{\colorlet{current}{.}\begingroup #5\endgroup\color{#1}#3#4\relax\color{current}#6}}%
}
\makeatother

\newcommand{\cbfrac}[3][mymathcolor]{\cgenfrac{#1}{}{}{}{}{#2}{#3}}

\AtBeginDocument{
  \renewcommand{\setminus}{\mathbin{\backslash}}%
}

\usepackage{graphicx}

\usepackage[tracking=false]{microtype}
\newcommand*\GetTextWidth[3][\normalfont]{{#1%
    \settowidth{#2}{abcdefghijklmnopqrstuvwxyz}%
    \setlength{#2}{0.03193#2}%
    \addtolength{#2}{0.44961pt}%
    \setlength{#2}{#3#2}%
    \global#2=#2}}
\newlength\bringhurstwdt
\GetTextWidth{\bringhurstwdt}{100}
\usepackage{geometry}
\usepackage{lipsum} 
\usepackage{caption}
\usepackage[position=b]{subcaption}
\usepackage{xfrac}
\usepackage{multicol}
\usepackage{multirow}
\usepackage[inline]{enumitem}
\usepackage{ellipsis}
\usepackage{float}
\usepackage{booktabs}
\usepackage{blkarray,bigstrut}
\usepackage{comment}
\usepackage{soul}
\usepackage[normalem]{ulem}
\usepackage{array}
\usepackage{marginnote}
\usepackage{wrapfig}
\usepackage[export]{adjustbox}
\usepackage{stmaryrd}
\usepackage{accents}
\usepackage{mleftright}
\usepackage{cancel}
\usepackage{centernot}

\usepackage{polynom}
\usepackage{tasks}
\usepackage{listings}
\lstdefinestyle{sagecode}{
    keywordstyle=\color{blue},
    commentstyle=\color{green!50!black},
    numberstyle=\tiny,
    basicstyle=\ttfamily,
    language=Python,
    morekeywords={%
        arrow,
        plot,
        parametric_plot,
        text,
        points},
    deletekeywords=[2]{min,max},
    xleftmargin=-38pt
}

\usepackage{varwidth}

\usepackage{tikz}
\usepackage{pgfplots}
\usetikzlibrary{arrows, arrows.meta, decorations.markings}
\usetikzlibrary{positioning}
\usetikzlibrary{calc,intersections,through,backgrounds}
\usetikzlibrary{quotes}
\usetikzlibrary{angles}
\usetikzlibrary{bending}
\usetikzlibrary{tikzmark}
\pgfplotsset{compat=1.18}
\usepgfplotslibrary{groupplots}

\usepackage{anyfontsize}
\usepackage{titlesec} 

\usepackage{stackengine}

\usepackage[multiple]{footmisc}

\usepackage{empheq}
\usepackage[linesnumbered,ruled,vlined]{algorithm2e}

\SetCommentSty{mycommfont}

\SetKwInput{KwInput}{Input}                
\SetKwInput{KwOutput}{Output}              

\theoremstyle{definition}
\newtheorem{proposition}{Proposition}
\newtheorem{conjecture}{Conjecture}
\newtheorem{theorem}{Theorem}
\newtheorem{lemma}{Lemma}

\newcommand{\F}{\mathbb{F}}

\newcommand{\Z}{\mathbb{Z}}

\newcommand{\N}{\mathbb{N}}

\makeatletter
\DeclareFontFamily{U}{tipa}{}
\DeclareFontShape{U}{tipa}{m}{n}{<->tipa10}{}
\newcommand{\arc@char}{{\usefont{U}{tipa}{m}{n}\symbol{62}}}%

\newcommand{\arc}[1]{\mathpalette\arc@arc{#1}}

\newcommand{\arc@arc}[2]{%
  \sbox0{$\m@th#1#2$}%
  \vbox{
    \hbox{\resizebox{\wd0}{\height}{\arc@char}}
    \nointerlineskip
    \box0
  }%
}
\makeatother
\renewcommand{\le}{\leqslant}
\usepackage[most]{tcolorbox}
\usepackage{hyperref}
\hypersetup{
    colorlinks,
    linkcolor={red!50!black},
    citecolor={blue!50!black},
    urlcolor={blue!80!black},
    bookmarksdepth=2
}
\makeatletter
\def\toclevel@part{-1}
\def\toclevel@chapter{0}
\def\toclevel@section{1}
\def\toclevel@subsection{2}
\def\toclevel@subsubsection{3}
\makeatother
\usepackage{bookmark}
\bookmarksetup{open,numbered,depth=2}
\makeatletter
\providecommand*\toclevel@part{-1}
\providecommand*\toclevel@chapter{0}
\providecommand*\toclevel@section{1}
\providecommand*\toclevel@subsection{2}
\providecommand*\toclevel@subsubsection{3}
\makeatother
\tcbuselibrary{breakable}

\newtcbtheorem[]{mytheo}{Theorem}%
    {%
        breakable,%
        colback=green!5,%
        colframe=green!35!black,%
        fonttitle=\bfseries,
        breakable
    }%
    {th}

\newtcbtheorem[]{myaxiom}{Axiom}%
    {%
        breakable,%
        colback=orange!5,%
        colframe=orange!35!black,%
        fonttitle=\bfseries,
        breakable
    }%
    {axiom}

\newtcbtheorem[]{myconj}{Conjecture}%
    {%
        breakable,%
        colback=magenta!5,%
        colframe=magenta!35!black,%
        fonttitle=\bfseries,
        breakable
    }%
    {conj}

\newtcbtheorem[]{myprop}{Proposition}%
    {%
        breakable,%
        colback=green!5,%
        colframe=green!35!black,%
        fonttitle=\bfseries,
        breakable
    }%
    {prop}

\newtcbtheorem[]{mylemma}{Lemma}%
    {%
        breakable,%
        colback=green!5,%
        colframe=green!35!black,%
        fonttitle=\bfseries,
        breakable
    }%
    {lemma}

\newtcbtheorem[]{mycor}{Corollary}%
    {%
        breakable,%
        colback=pink!5,%
        colframe=pink!35!black,%
        fonttitle=\bfseries,
        breakable
    }%
    {corol}

\newtcbtheorem[]{mydef}{Definition}%
    {%
        breakable,%
        colback=orange!5,%
        colframe=orange!35!black,%
        fonttitle=\bfseries,
        breakable
    }%
    {df}

\newtcbtheorem[]{myex}{Example}%
    {%
        breakable,%
        colback=magenta!5,%
        colframe=magenta!35!black,%
        fonttitle=\bfseries,
        breakable
    }%
    {ex}

\newtcbtheorem[]{myexer}{Exercise}%
    {%
        breakable,%
        colback=magenta!5,%
        colframe=magenta!35!black,%
        fonttitle=\bfseries,
        breakable
    }%
    {exer}

\newtcbtheorem[]{myprob}{Problem}%
    {%
        breakable,%
        colback=magenta!5,%
        colframe=magenta!35!black,%
        fonttitle=\bfseries,
        breakable
    }%
    {prob}

\newtcblisting[auto counter]{mylisting}[2][]{%
    sharp corners,
    fonttitle=\bfseries,
    colframe=gray,
    detach title,
    listing only,
    breakable,
    listing options={style=sagecode},
}
\newcommand{\tc}{\textcolor{black}{\text{,}}}

\newcommand{\td}{\textcolor{black}{\text{.}}}

\newcommand{\mt}[1]{\textcolor{black}{\text{#1}}}

\usepackage[maxbibnames=6,sorting=anyt,giveninits=true]{biblatex}
\title{Isomorphism Criterion of Monomial Digraphs over Prime Fields}

\author{
    \begin{tabular}{ccc}
        Alexander M.\ Kodess & Felix Lazebnik & Mikhail Muzychuk\\
        \small{Department of Mathematics} &
        \small{Department of Math.\ Sciences} &
        \small{Department of Mathematics}\\
        \small{Farmingdale State College} &
        \small{University of Delaware} &
        \small{Ben-Gurion University of the Negev}\\
        \small{Farmingdale, NY, USA} &
        \small{Newark, DE, USA} &
        \small{Beer-Sheva, Israel}\\
        \href{mailto:kodessa@farmingdale.edu}%
        {\small{\texttt{kodessa@farmingdale.edu}}} &
        \href{mailto:fellaz@udel.edu}%
        {\small{\texttt{fellaz@udel.edu}}} &
        \href{mailto:muzychuk@bgu.ac.il}%
        {\small{\texttt{muzychuk@bgu.ac.il}}}
    \end{tabular}
}

\begin{document}

\maketitle

\textit{This article is dedicated to the memory of Iosif Yakovlevich Rivkind (1952--2026), an
extraordinary teacher of mathematics and computer science.}

\begin{abstract}
For any 
Galois field
$ \F_q $ 
with 
$ q $ 
elements and any positive integers 
$ m $ 
and 
$ n $, 
the directed graph
$ D\mleft( q ; m , n \mright) $ 
has vertex set 
$ \F_q \times \F_q $, 
and there is an arc from vertex
$ \mleft( x_1 , x_2 \mright) $ 
to vertex
$ \mleft( y_1 , y_2\mright) $ 
if and only if
$ x_2 + y_2 = x_1^m y_1^n $.
It was conjectured in earlier work 
that two digraphs 
$ D\mleft( q ; m_1 , n_1 \mright) $
and
$ D\mleft( q ; m_2 , n_2\mright) $
are isomorphic if and only if there exists an integer 
$ k $
relatively prime to  
$ q - 1 $ 
such that 
$ m_2 \equiv  k m_1 $
and
$n_2 \equiv  k n_1 $, 
where both congruences are modulo 
$ q - 1 $. 
We prove this conjecture over prime fields and 
construct an infinite family of counterexamples 
over extension fields.
\end{abstract}

\section{Introduction}
Let
$ \F_q $
denote 
the finite field with
$ q = p^e $ elements, where $ p $ is prime. 
For a subset $ A $ of $ \F_q $ we will denote 
$ A \setminus \mleft\{ 0 \mright\} $ 
by $ A^{ \mystar } $. 
Our reference on finite fields is
Lidl and Niederreiter~\cite{lidl1997finite}.

Since 
$ x^{ q - 1 } = 1 $ 
for every 
$ x \in \F_q^{ \mystar } $, 
when we write 
$ x^t $ 
for 
$ x \in \F_q $ 
and 
$ t \in \N $,  
we may assume that 
$ 1 \le t \le q - 1 $, 
unless stated otherwise. 
Also, if 
$ \gcd\mleft( q - 1 , t \mright) = 1 $, 
then there exists a unique 
$ t' $, 
$ 1 \le t' \le q - 1 $,  
such that 
$ t t' \equiv 1 \mybmod q - 1 $. 
We will denote such 
$ t' $ 
by 
$ t^{ -1 } $. 
Unless stated otherwise, all congruences are modulo $ q - 1 $.
Whenever $ q $ is fixed and $ x $ is an integer, we shall write
$ \ovl{ x } $ 
for
$ \gcd(x,  q-1 ) $,  
and for the system of congruences 
$ km_1 \equiv m_2 \mybmod q - 1 $ 
and 
$ k n_1 \equiv n_2 \mybmod q - 1 $, 
we will just write 
$ k \left( m_1 , n_1 \right) \equiv \left( m_2 , n_2 \right) $.

In this paper we consider a class of directed graphs.
A \emph{directed graph} (or just \emph{digraph}) 
$ D = ( V , A ) $ 
consists of
a finite non-empty vertex set 
$ V = V( D ) $
and a finite arc set 
$ A = A( D ) $ 
of ordered pairs of not necessarily distinct vertices. 
Note that our digraphs allow loops, but not parallel arcs. 

For a digraph 
$ D $, 
we will denote by 
$ D^{ \operatorname{con} } $
its \emph{converse digraph} 
obtained by reversing every arc of 
$ D $.

If 
$ u , v \in V $ and $(u,v)$ is an arc of $D$, we say that 
the arc $(u,v)$, also denoted as $u \to v$,   is from $u$ to $v$, 
and in this case 
$ v $ 
is called an \emph{out-neighbor} of 
$ u $, 
and 
$ u $ 
is called an \emph{in-neighbor} of 
$ v $. 
For a vertex 
$ x $ 
of 
$ D $, 
by 
$ D^{ + } ( x ) $ 
(resp.\ $ D^{ - } ( x ) $), 
we denote the set of all 
out-neighbors (resp.\ in-neighbors) of $ x $. 
We call the set $D^{ + } (x)$ (resp.\ $D^{ - }  (x)$) 
the \emph{out-neighborhood} (resp.\ the \emph{in-neighborhood}) 
of $x$ in $D$.  
Equivalently, 
\[
    D^{ + }( x )
    = 
    \mleft\{
        v \in V 
        :
        ( x , v ) \in A
    \mright\}
    \qquad 
    \mt{and}
    \qquad 
    D^{ - }( x )
    = 
    \mleft\{
        v \in V
        :
        ( v , x ) \in A
    \mright\}%
    \td
\]
For all graph- or digraph-theoretic terms not explicitly defined
here we refer the reader
to
Bang-Jensen and Gutin~\cite{bang2008digraphs,bang2018classes}
or to  Gross, Yellen, and Zhang~\cite{grossgraphth2013}. 

Let 
$ q $ 
be a prime power and let 
$ m , n $ 
be positive integers. 
The digraphs 
$ D( q ; m , n ) $ 
are the main objects of our study, and they are defined as follows.
The vertex set of 
$ D( q ; m , n ) $ 
is 
$ \F_q^2 = \F_q \times \F_q $, 
and for vertices 
$ x =( x_1 , x_2 ) $
and 
$ y =( y_1 , y_2 ) $, 
$( x , y) $
 is
an arc in 
$ D( q ; m , n ) $ 
if and only if
\begin{equation}
    \label{eqn:dqmn_adjacency_condition}
    x_2 + y_2 = x_1^m y_1^n%
    \td
\end{equation}

\begin{figure}[H]
    \centering
        \begin{tikzpicture}

        \tikzset{vertex/.style = {shape=circle,draw,inner sep=2pt,minimum size=.5em, scale = 1.0},font=\sffamily\scriptsize\bfseries}
        \tikzset{edge/.style = {->,> = triangle 45}}
        \node[vertex] (a) at  (2,0) {$(1,0)$};
        \node[vertex] (b) at  (5,0) {$(0,0)$};
        \node[vertex] (c) at  (8,0) {$(2,0)$};

        \node[vertex] (d) at  (0,-2) {$(2,1)$};
        \node[vertex] (e) at  (2,-2) {$(1,1)$};
        \node[vertex] (f) at  (8,-2) {$(2,2)$};
        \node[vertex] (g) at  (10,-2) {$(1,2)$};

        \node[vertex] (h) at  (2,-4) {$(0,2)$};
        \node[vertex] (i) at  (8,-4) {$(0,1)$};

        \draw[edge] (a) to (b);
        \draw[edge] (a) to (d);
        \draw[edge] (a) to (e);

        \draw[edge] (b) to (a);
        \draw[edge] (b) to (c);

        \draw[edge] (c) to (b);
        \draw[edge] (c) to (f);
        \draw[edge] (c) to (g);

        \draw[edge] (d) to (e);
        \draw[edge] (d) to (h);

        \draw[edge] (e) to (a);
        \draw[edge] (e) to (c);
        \draw[edge] (e) to (h);

        \draw[edge] (f) to (a);
        \draw[edge] (f) to (c);
        \draw[edge] (f) to (i);

        \draw[edge] (g) to (f);
        \draw[edge] (g) to (i);

        \draw[edge] (h) to (d);
        \draw[edge] (h) to (e);
        \draw[edge] (h) to (i);

        \draw[edge] (i) to (f);
        \draw[edge] (i) to (g);
        \draw[edge] (i) to (h);

        \path
        (b) edge [->,>={triangle 45[flex,sep=0pt]},loop,out=240,in=270,looseness=6.5] node {} (b);
        \path
        (d) edge [->,>={triangle 45[flex,sep=0pt]},loop,out=160,in=130,looseness=8] node {} (d);
        \path
        (g) edge [->,>={triangle 45[flex,sep=0pt]},loop,out=20,in=50,looseness=8] node {} (g);
        \end{tikzpicture}
        \caption{The monomial digraph $D(3;1,2)$: $x_2+y_2 = x_1y_1^2$.}
        \label{dig_312}
\end{figure}

In view of the presence of a monomial in this adjacency 
condition~\eqref{eqn:dqmn_adjacency_condition}, 
we shall 
refer to the digraphs 
$ D( q ; m , n ) $ 
as 
\emph{monomial digraphs}.

Whenever the values of 
$ q $, 
$ m $, 
and 
$ n $ 
are fixed,
we shall often write
$ D $
for
$ D( q; m, n ) $.

Clearly, for 
$ D = D( q ; m , n ) $ 
and 
$ \mleft( a , b \mright) \in V\mleft( D \mright) $, 
simplifying the notation, 
we have 
\begin{align}
    D^{ + }( (a , b) ) 
    = 
    D^{ + }( a , b ) 
    &=
    \bigl\{\,
        ( 
            x 
            , 
            - b + a^m x^n 
        )
        :
        x \in \F_q
    \,\bigr\}
    \;\;
    \text{\rm and}
    \label{eqn:general_neighborhood}
    \\
    D^{ - }( (a , b) )
    = 
    D^{ - }( a , b ) 
   &= 
    \bigl\{\,
        ( x,-b + x^m a^n )
        :
        x \in \F_q
    \,\bigr\}%
    \td
    \nonumber
\end{align}

The following properties of 
$ D = D( q ; m , n ) $ 
hold for all parameter values. 
    \begin{itemize}
        \item
            The arc from vertex 
            $ ( 0 , 0 ) $ 
            to itself is present in 
            $ D $, 
            and so 
            $ D $ 
            always contains at least one loop.
        \item
            The in-degree and out-degree of every vertex
            in 
            $ D $ 
            are both equal to  
            $ q $. 
            Also, 
            $ D $ 
            has 
            $ q^2 $ 
            vertices and 
            $ q^3 $ 
            arcs.
    \end{itemize}
    
The digraphs
$ G $
and
$ H $
are called \emph{isomorphic}
(denoted by $ G \cong H $)
if there exists a bijection
$ \varphi \colon  V (G) \to V (H) $
such that, for every $x,y \in V(G)$, 
$ \mleft( x , y \mright) \in A( G ) $
if and only if
$  \mleft( \varphi( x ) , \varphi( y ) \mright) \in A( H ) $.
The mapping
$ \varphi $
is called an \emph{isomorphism} from 
$ G $ 
to 
$ H $.
As we shall see, some of the digraphs 
$ D( q ; m , n ) $
belong to the same isomorphism class.
Note that as
$ m $
and
$ n $
vary from
$ 1 $
to
$ q - 1 $,
for any fixed
$ q $
we obtain
$ ( q - 1 )^2 $
monomial digraphs (each with vertex set
$\F_q \times \F_q $).
The problem of classifying all these 
$ \mleft( q - 1 \mright)^2 $ 
digraphs
up to isomorphism was initiated in Kodess~\cite{kod14},  and it is the primary goal of this paper.
  
For a fixed 
$ q $, 
let 
$ D_i = D( q; m_i , n_i ) $, 
$ i = 1,2 $.
The following proposition gives a sufficient condition for 
$ D_1 $ 
and 
$ D_2 $ 
to be isomorphic. 
\begin{proposition}[\cite{KL2017}]
\label{proposition:sufficiency} 
Let 
$ q $ 
be a prime power. 
If there exists 
$ k $, 
$ \ovl{k} = 1 $, 
such that  
$ 
\left( m_2 , n_2 \right) 
\equiv 
k
\left( m_1 , n_1 \right) 
$, 
then  
$   
D_1\cong D_2
$.
\end{proposition} 

Kodess and Lazebnik \cite{KL2017} conjectured
that the converse of 
Proposition~\ref{proposition:sufficiency} 
is true, that is, the following holds. 
\begin{conjecture}
\label{false_conjecture} 
Let 
$ q $ 
be a prime power. 
If $ D_1 \cong D_2 $, then there exists 
$ k $, 
$ \ovl{k} = 1 $, 
such that  
$ 
    \mleft( m_2 , n_2 \mright) 
    \equiv 
    k
    \mleft( m_1 , n_1 \mright) 
$.
\end{conjecture}
It turned out to be false in general:
digraphs 
$ D( 25 ; 3 , 6 ) $ 
and 
$ D( 25 ;3 , 18 ) $ 
are isomorphic, but there is no integer 
$k$ such that $ (3,18) \equiv  k (3,6) \mybmod 24 $. 
(We acknowledge our previously mistaken computation in~\cite{KL2017}, 
which overlooked this pair.) 
This counterexample was extended to an infinite family which we describe and prove later.
Nevertheless, we prove the conjecture in the special case of prime $ q $.
\begin{theorem}
\label{theorem:prime_classification}
Let
$ q = p $
be prime. 
If
$ D_1 \cong D_2 $,
then there exists an integer
$ k $, 
$ \ovl{k}=1 $,
such that
$
    \mleft(m_2,n_2\mright)
    \equiv
    k
    \mleft(m_1,n_1\mright)%
    \td
$
\end{theorem}
This paper is organized as follows. 
In Section~\ref{section:auxiliary} 
we present some known conditions on parameters 
$ m_i,n_i $, 
$ i = 1 , 2 $, 
which are necessary for 
$ D_1 \cong D_2 $, 
as well as some properties of isomorphisms 
of monomial digraphs.
In Section~\ref{section:main_proofs} 
we prove Theorem~\ref{theorem:prime_classification}.
In Section~\ref{section:infinite_counterexamples} 
we give an infinite class of counterexamples to the converse of 
Proposition~\ref{proposition:sufficiency}, 
and in Section~\ref{section:concluding_remarks} 
we present concluding remarks. 
%
\section{Preliminary Results}
\label{section:auxiliary}
For several statements in this section, we refer the reader to previously published articles. We include short proofs and comments for a few results in order to illustrate the techniques used later.

Now we present some necessary conditions that the parameters of 
isomorphic monomial digraphs must satisfy. 
The statements in Proposition~\ref{proposition:cond_on_parameters}
first appeared in~\cite{KL2017}. 

\begin{proposition}[\cite{KL2017}]
\label{proposition:cond_on_parameters}
Let $ q $ be an odd prime power.
If
$ D_1 \cong D_2 $, 
then:
\begin{enumerate}[label=(\roman*)]
\item 
    \label{item:equal_bars}
    $ \ovl{ m_1 } = \ovl{ m_2 } $
    and
    $ \ovl{ n_1 } = \ovl{ n_2 } $.
\item 
    \label{item:equal_parameters}
    If   
    $ m_1 = n_1 $,
    then
    $ m_2 = n_2 $
    and
    $ \ovl{ m_1 } = \ovl{ m_2 } $.
\end{enumerate} 
\end{proposition}

To simplify notation, 
for an isomorphism 
$ \varphi \colon D_1 \to D_2 $ 
we shall write 
$ \varphi\mleft( x , y \mright) $ 
for 
$ \varphi\mleft( \mleft( x , y \mright) \mright) $, 
where 
$ \mleft( x , y \mright) \in V\mleft( D_1 \mright) $.

We will first need  the following lemma, which 
appeared in
Viglione~\cite{ViglioneThesis}
and 
in Dmytrenko, Lazebnik, and Viglione~\cite{dlv}. 
\begin{lemma}[\cite{ViglioneThesis}, \cite{dlv}]
\label{lemma:dirichlet}
\it{
Let 
$ a $, 
$ b $, 
and 
$ c \neq 0 $ 
be integers. 
If 
$ \gcd( a , c ) = \gcd( b , c ) $, 
then  there exists an integer 
$ k $ 
such that 
$ \gcd( k , c ) = 1 $ 
and 
$ k a  \equiv b \mybmod c $.
}
\end{lemma}

The two statements above imply 
Conjecture~\ref{false_conjecture} in 
the case when 
$ m_1 = n_1 $ (or $ m_2 = n_2 $).
\begin{proposition}
\label{proposition:m_equal_n}
Let $ q $ be an odd prime power. 
If 
$ D_1 \cong D_2 $ 
and 
$ m_1 = n_1 $, 
then 
$ \mleft( m_2 , n_2 \mright) \equiv k \mleft( m_1 , n_1 \mright) $ 
for some 
$ k $, 
$ \ovl{ k } = 1 $.
\end{proposition}
\begin{proof}
Since $D_1\cong D_2$, 
part~(ii) of Proposition~\ref{proposition:cond_on_parameters} 
implies that 
$ m_2 = n_2 $ 
and 
$ \ovl{ m_1 } = \ovl{ m_2 } $. 
Then, by  
Lemma~\ref{lemma:dirichlet} with 
$ a = m_1 $, 
$ b = m_2 $, 
$ c = q - 1 $,   
$\mleft( 
        m_2 , m_2 
    \mright)
    \equiv 
    k
    \mleft( 
        m_1 , m_1 
    \mright) = k
    \mleft( 
        m_1 , n_1 
    \mright)$
for some 
$ k $, 
$ \ovl{ k } = 1 $. 
Therefore 
$ \mleft( m_2 , n_2 \mright) \equiv k \mleft( m_1 , n_1 \mright) $. 
\end{proof} 

The following properties of isomorphisms of monomial digraphs 
will be used later.

A polynomial
$ h \in \F_q [ X ] $
is called a
\emph{permutation polynomial}
on
$ \F_q $
if 
$ a \mapsto h\mleft( a \mright) $ 
is a bijection of 
$ \F_q $.

\begin{proposition}
\label{proposition:f_and_g}
Let 
$ q $ 
be 
an odd 
prime power, 
and suppose that 
$ m_1 \neq n_1 $, 
and 
$ D_1 \cong D_2 $
with an isomorphism given in the form
\begin{equation*}
    \varphi
    \colon
    V( D_1 )
    \longrightarrow
    V( D_2 )%
    \tc
    \qquad
    ( x , y )
    \longmapsto
    (
        f( x , y ),
        g( x , y )
    )
\end{equation*}
for some
$ f , g \in \F_q[ X , Y ] $
of degree at most
$ q - 1 $
in each of the variables.
Then the following statements hold.
\begin{enumerate}[label=(\roman*)]
\item 
    \label{prop:g_is_odd}
    $ g $
    is a permutation polynomial of indeterminate
    $ Y $
    only, 
    and 
    $ g \mleft( -y \mright) = -g\mleft( y \mright) $ 
    for all 
    $ y \in \F_q $. 
    In particular, 
    $ g\mleft( 0 \mright) = 0 $.
\item 
    \label{prop:f_x_0_is_0}
    $ f( x , y ) = 0 $
    if and only if
    $ x = 0 $. 
\item
    \label{prop:f_one_var_permutation}
    For every fixed           
    $ y \in \F_ q $, 
    $ x \mapsto f( x , y ) $ 
    is a permutation of $ \F_q $.
\item 
    \label{prop:normalized_isomorphism}
    There exists a \emph{normalized} isomorphism 
    $ 
        \hat{ \varphi } \colon D_1 \to D_2 
    $ 
    such 
    that 
    \begin{equation}
        \label{eqn:varphi1_10_mapsto_10}
        \hat{ \varphi }\mleft( 0 , 0 \mright) 
        =
        \mleft( 0 , 0 \mright)%
        \qquad 
        \mt{and}
        \qquad 
        \hat{ \varphi }\mleft( 1 , 0 \mright) 
        =
        \mleft( 1 , 0 \mright)%
        \td
    \end{equation}
\end{enumerate}
\end{proposition}

\begin{proof}
Parts~\ref{prop:g_is_odd} 
and~\ref{prop:f_x_0_is_0} 
were proved 
in~\cite[Theorem~2.1]{KL2017}, 
and part~\ref{prop:f_one_var_permutation} 
follows from part~\ref{prop:g_is_odd} 
immediately. 
We, therefore, only prove 
part~\ref{prop:normalized_isomorphism}.

Let 
$ a = f\mleft( 1 , 0 \mright) \neq 0  $. 
Then 
$
    \varphi \mleft( 0 , 0 \mright)
    =
    \mleft( 0 , 0 \mright)%
$ 
and 
$
    \varphi \mleft( 1 , 0 \mright)
    =
    \mleft( a , 0 \mright)%
$.
The mapping 
$
    \hat{ \varphi }
    =
    \psi_{ a^{ - 1 } } \circ \varphi 
$, 
where
$
    \psi_c
    \mleft( x , y \mright)
    =
    \mleft(
        c x,c^{ m_2 + n_2 } y
    \mright)
$ 
is an obvious automorphism of
$ D_2 $, 
satisfies 
conditions~\eqref{eqn:varphi1_10_mapsto_10}.
\end{proof}

The following lemma will be used in the proof of 
Theorem~\ref{theorem:prime_classification}.
\begin{lemma}
\label{lemma:power_identities}
Let $ q $ be a prime power. 
Suppose there exists a permutation 
$ F $ 
of 
$ \F_q^{ \mystar } $ 
such that 
\begin{equation}
    \label{eqn:power_identities}
    F\mleft( x \mright)^{ m_2 } = x^{ m_1 } 
    \qquad 
    \mt{and}
    \qquad 
    F\mleft( x \mright)^{ n_2 } = x^{ n_1 } 
    \quad 
    \mt{for every $ x \in \F_q^{ \mystar } $}%
    \td 
\end{equation}
Then there exists an integer 
$ k $, 
$ \ovl{ k } = 1 $, 
such that 
$
    \mleft( m_2 , n_2 \mright) 
    \equiv 
    k
    \mleft( m_1 , n_1 \mright) 
$.
\end{lemma} 
\begin{proof} 
Since $ F $ is a permutation of $ \F_q^{ \mystar } $, 
identities~\eqref{eqn:power_identities} imply that 
\begin{equation}
    \label{eqn:equal_power_sets}
    \mleft\{ 
        \mleft( x^{ m_1 } , x^{ n_1 } \mright)
        :
        x \in \F_q^{ \mystar } 
    \mright\}
    =
    \mleft\{ 
        \mleft( x^{ m_2 } , x^{ n_2 } \mright)
        :
        x \in \F_q^{ \mystar } 
    \mright\}%
    \td
\end{equation}
Denote by $ \widehat{ m_i } $ and $ \widehat{ n_i } $
the residue classes of $ m_i $ and $ n_i $ in 
$ \Z_{ q - 1 } $. 
Fixing a generator 
$ \omega $ 
of 
$ \F_q^{ \mystar } $ 
and using the isomorphism 
$ \Z_{ q - 1 }^2 \to \mleft( \F_q^{ \mystar } \mright)^2 $ 
given by 
$ 
    \mleft( \widehat{ a } , \widehat{ b } \mright) 
    \mapsto 
    \mleft( \omega^{ a } , \omega^{ b } \mright)  
$, 
we conclude from~\eqref{eqn:equal_power_sets} 
that the cyclic subgroups of 
$ \Z_{ q - 1 }^2 $ 
generated by 
$ \mleft( \widehat{ m_1 } , \widehat{ n_1 } \mright) $ 
and 
$ \mleft( \widehat{ m_2 } , \widehat{ n_2 } \mright) $ 
are equal. 
Let $ M $ be their common order.

There exists $ \ell \in \Z $ such that 
$ \gcd\mleft( M , \ell \mright) = 1 $ 
and 
$ 
    \mleft( \widehat{ m_2 } , \widehat{ n_2 } \mright) 
    =
    \ell 
    \mleft( \widehat{ m_1 } , \widehat{ n_1 } \mright)
$. 
Put $ d = \mleft( q - 1 \mright) / M $. 
Then 
\[
    \gcd\mleft( d \ell , q - 1 \mright) 
    = 
    \gcd\mleft( d \ell , d M \mright) 
    = 
    d \gcd\mleft( \ell , M \mright) 
    =
    d
    =
    \gcd\mleft( d , q - 1 \mright)%
    \td
\]
By Lemma~\ref{lemma:dirichlet}, there exists an integer 
$ k $, 
$ \ovl{ k } = 1 $, 
such that 
$ k d \equiv d \ell $, 
and so 
$ k \equiv \ell \mybmod M $. 
Since the order of 
$ \mleft( \widehat{ m_1 } , \widehat{ n_1 } \mright) $ in 
$ \Z_{ q - 1 }^2 $
is 
$ M $, 
we have 
$
    k
    \mleft( \widehat{ m_1 } , \widehat{ n_1 } \mright) 
    =
    \ell
    \mleft( \widehat{ m_1 } , \widehat{ n_1 } \mright) 
    =
    \mleft( \widehat{ m_2 } , \widehat{ n_2 } \mright) 
$. 
This gives the desired congruence 
$
    k
    \mleft( m_1 , n_1 \mright)
    \equiv 
    \mleft( m_2 , n_2 \mright)
$.
\end{proof} 
\section{Proof of Theorem~\ref{theorem:prime_classification}}
\label{section:main_proofs}
Define 
$  A_{ m } $ 
to be the group of all 
$ m $-th nonzero powers of 
$ \F_q $, 
that is, let 
$ A_m  = \mleft\{ x^m  : x \in \F_q^{ \mystar }  \mright\} $.  
Whenever 
$ \ovl{ m } = \ovl{ n } $, 
$ 
    \mleft\lvert 
        A_m 
    \mright\rvert 
    =
    \mleft\lvert 
        A_n
    \mright\rvert 
    =
    \mleft( q - 1 \mright) / \ovl{ m }
$ 
(see~\cite[p.~7]{lidl1997finite}), 
and, since $ \F_q^{ \mystar } $ is cyclic, 
we have 
$ A_m = A_n $. 
We now prove Theorem~\ref{theorem:prime_classification}. 
\begin{proof} 
The case $ p = 2 $ is trivial, so we assume that $ p $ is odd. 

If $ m_1 = n_1 $,
then the conclusion follows from 
Proposition~\ref{proposition:m_equal_n}.
Therefore, for the rest of the proof we assume that
$ m_1 \neq n_1 $. 

By Proposition~\ref{proposition:cond_on_parameters}, 
part~\ref{item:equal_bars}, 
$ \ovl{ m_2 } = \ovl{ m_1 } $ 
and 
$ \ovl{ n_2 } = \ovl{ n_1 } $.
After passing to the converse digraphs, if necessary, we may assume
that
$ \ovl{ m_1 } \le \ovl{ n_1 } $. 

Let 
$ \varphi \colon D_1 \to D_2 $, 
$ 
    \mleft( x , y \mright) 
    \mapsto 
    \mleft( f\mleft( x , y \mright) , g \mleft( y \mright) \mright) 
$ 
be a normalized isomorphism.

\medskip

Case 1:  
$ \ovl{ n_1 } > 1 $. 

If 
$ \ovl{ n_1 } = \ovl{ n_2 } = p - 1 $, 
then 
$ n_1 = n_2 = p - 1 $. 
In this case Lemma~\ref{lemma:dirichlet}
gives an integer
$ k $,
$ \ovl{ k } = 1 $,
such that
$ k m_1\equiv m_2 $ 
and 
$ k n_1 \equiv n_2 \equiv 0 $.
The conclusion follows in this case. 
We may assume now that 
$ 1 < \ovl{ n_1 } = \ovl{ n_2 } < p - 1 $.

\medskip

Our goal now is to show that 
$ g $ is linear, namely  
that there exists 
$ \lambda \in \F_p^{ \mystar } $ 
such that 
$ g\mleft( z \mright) = \lambda z $ 
for all
$ z \in \F_p$. 

Let 
$ A = A_{ n_1 } = A_{ n_2 } $. 
Then  
$ 1 < \mleft\lvert A \mright\rvert < p - 1 $. 

Consider an arc   
$ 
    \mleft( 1 , -z \mright) 
    \to 
    \mleft( x , z + x^{ n_1 } \mright)
$ 
in 
$ D_1 $. 
Applying $ \varphi $, and using $ g\mleft( -z \mright) = -g\mleft( z \mright) $ 
for all 
$ z \in \F_p $ (see Proposition~\ref{proposition:f_and_g}, 
part~\ref{prop:g_is_odd}), 
we obtain 
\begin{equation} 
\label{eqn:g_z_plus_x_n1}
    g\mleft( z + x^{ n_1 } \mright) 
    =
    g\mleft( z \mright) 
    +
    f\mleft( 1 , -z \mright)^{ m_2 } 
    f\mleft( x , z + x^{ n_1 } \mright)^{ n_2 }
    \td
\end{equation}
Because the first coordinate of the out-neighborhood of 
every vertex in 
$ D_2 $ 
runs over all 
$ \F_p $ (see~\eqref{eqn:general_neighborhood}), 
we conclude that 
$
    \mleft\{
        f\mleft( x , z + x^{ n_1 } \mright) 
        :
        x \in \F_p
    \mright\}
    =
    \F_p
$
for every 
$ z \in \F_p $. 
Taking 
$ 
    c\mleft( z \mright)
    =
    f\mleft( 1 , -z \mright)^{ m_2 } 
    \ne 0
$ 
(Proposition~\ref{proposition:f_and_g}, 
part~\ref{prop:f_x_0_is_0}),  
we can rewrite~\eqref{eqn:g_z_plus_x_n1} as 
\begin{equation}
    \label{eqn:g_z_plus_V}
    g\mleft( z + A \mright) 
    =
    g\mleft( z \mright) 
    + 
    c\mleft( z \mright) 
    A%
    \td
\end{equation}
Fix $ v \in A $. 
By~\eqref{eqn:g_z_plus_V}, the nonzero difference 
$ 
    g\mleft( z + v \mright) 
    -
    g\mleft( z \mright) 
    =
    g\mleft( -z \mright) 
    -
    g\mleft( -z - v \mright) 
$ 
belongs to both cosets 
$ c\mleft( z \mright) A $ 
and 
$ c\mleft( - z - v \mright) A $ 
of 
$ A \le \F_p^{ \mystar } $. 
Intersecting cosets are equal, and so 
$ c\mleft( z \mright) A = c\mleft( - z - v \mright) A $. 

Fix 
$ w \in A $, 
$ w \neq v $. 
Such 
$ w $ 
exists since 
$ \mleft\lvert A \mright\rvert > 1 $.
By an argument similar to the one we just used, we obtain 
$ 
    c\mleft( -z - v \mright) A
    = 
    c\mleft(\mleft( z + v \mright) - w\mright) A
    = 
    c\mleft( z + \mleft( v - w \mright) \mright) A
$. 
Therefore 
$ 
    c\mleft( z \mright) A
    = 
    c\mleft( z + \mleft(v - w \mright) \mright) 
    A
$ 
for all 
$ z\in \F_p $, 
and a fixed 
$ v - w \neq 0 $. 
Because every nonzero element of $ \F_p $ additively generates it, 
this implies 
$ c\mleft( z \mright) A = c\mleft( 0 \mright) A = A $. 
Then, from~\eqref{eqn:g_z_plus_V} we conclude that 
$ g\mleft( z + A \mright) = g\mleft( z \mright) + A $. 

This equality implies that 
$ g $ 
is an automorphism of a nontrivial Cayley digraph 
$ \operatorname{Cay} \mleft( \F_p , A \mright)$. 
The automorphism group of 
$ \operatorname{Cay} \mleft( \F_p , A \mright) $ 
is a permutation group of prime degree, transitive but not doubly transitive.  
It easily follows from a theorem of Burnside~\cite[p.~341]{burnside1911theory}, 
that such a permutation group is isomorphic to a subgroup of the affine group 
$ AGL_1\mleft( \F_p \mright) $. 
See also a more recent paper by 
Alspach~\cite{alspach1973point}
or by 
M\"{u}ller~\cite{Muller2005}. 
It turns out that this statement was essentially proven by Galois in 1831; 
see several references in 
Jones and Zvonkin~\cite{jones2023groups}.
Hence, 
$ g $ 
is affine, and 
$ g\mleft( 0 \mright) = 0 $ gives 
$ g\mleft( z \mright) = \lambda z $, $ \lambda \in \F_p^{ \mystar } $, 
as claimed. 

\medskip

We next prove the stronger identities 
\begin{equation}
\label{eqn:g_z_eq_z_f_powers}
g\mleft( z \mright) = z%
\tc 
\qquad 
f\mleft( x , y \mright)^{ m_2 } = x^{ m_1 }
\tc 
\qquad 
f\mleft( x , y \mright)^{ n_2 } = x^{ n_1 }%
\quad 
\mt{for all $ x , y , z \in \F_p $}%
\td
\end{equation}
As $\ovl{m_1} = \ovl{m_2}$, let 
$ U = A_{ m_1 } = A_{ m_2 } $.  
Consider  any 
$ a , x \in \F_p^{ \mystar } $, $y\in \F_p $. Let 
$ u = x^{ m_1 } \in U $ and  
$ v = a^{ n_1 } \in A $. Then 
 $
    \mleft(  x , y \mright) 
    \to 
    \mleft( a , u v - y \mright) 
$ 
and 
$
    \mleft( 1 , y + \mleft( 1 - u \mright) v \mright) 
    \to 
    \mleft( a , u v - y \mright) 
$ 
are arcs in 
$ D_1 $. 
Applying 
$ \varphi $ 
to both arcs and setting 
$ B\mleft( y \mright) = f\mleft( 1 , y \mright)^{ m_2 } $, 
we obtain 
\begin{align}
\lambda u v 
&=
f\mleft( x , y \mright)^{ m_2 } 
f\mleft( a , u v - y \mright)^{ n_2 }%
\tc
\label{eqn:lambda_u_v_first_arc}
\\
\lambda v 
&=
B\mleft( y + \mleft( 1 - u \mright) v \mright) 
f\mleft( a , u v - y \mright)^{ n_2 }%
\nonumber
\td
\end{align} 
Dividing these two equations gives 
\begin{equation} 
\label{eqn:f_x_y_m2_eq_uB}
f\mleft( x , y \mright)^{ m_2 } 
=
u
B\mleft( y + \mleft( 1 - u \mright) v \mright)%
\td
\end{equation} 
We claim that $ B $ is constant on $ \F_p $.
Since 
$ 
    \mleft\lvert U \mright\rvert 
    = 
    \mleft( p - 1 \mright) / \ovl{ m_1 } 
    > 
    1 
$ (indeed, as $\ovl{m_1} \le \ovl{n_1} < p-1$), 
we may choose  $x\in \F_p^*$ so that 
$ u = x^{m_1} \neq 1 $. 
Choose distinct $ v , w \in A $. 
From~\eqref{eqn:f_x_y_m2_eq_uB} it follows that
\[
    B\mleft(
        y + \mleft( 1 - u \mright) v 
    \mright)
    =
    B\mleft(
        y + \mleft( 1 - u \mright) w
    \mright)
    \qquad 
    \mt{for every $ y \in \F_p $.}
\] 
Put 
$ t = y + \mleft( 1 - u \mright) w $ 
and 
$ h = \mleft( 1 - u \mright) \mleft( v - w \mright) \neq 0 $. 
Then
\[
    B\mleft( t + h \mright) = B\mleft( t \mright) 
    \qquad 
    \mt{for every $ t \in \F_p $.}
\]
Iterating gives 
$ B\mleft( t + j h \mright) = B\mleft( t \mright) $ 
for all $ j $, $ 0 \le j \le p - 1 $. 
Hence, $ B $ is constant. 

As $ B\mleft( 0 \mright) = f(1,0)^{m_2} = 1 $, 
\eqref{eqn:f_x_y_m2_eq_uB} yields 
$ f\mleft( x , y \mright)^{ m_2 } = x^{ m_1 } $. 
This, together with~\eqref{eqn:lambda_u_v_first_arc}, 
yields 
$ f\mleft( a , b \mright)^{ n_2 } = \lambda a^{ n_1 } $ 
for all 
$ a , b \in \F_p $. 
Since $ f\mleft( 1 , 0 \mright) = 1 $, this forces 
$ \lambda = 1 $. 
All three claims in~\eqref{eqn:g_z_eq_z_f_powers} are proved. 

Define 
$ F\mleft( x \mright) = f\mleft( x , 0 \mright) $, 
$ x \in \F_p $, so 
$ F\mleft( x \mright)^{ m_2 } = x^{ m_1 } $ 
and 
$ F\mleft( x \mright)^{ n_2 } = x^{ n_1 } $. 
By 
Proposition~\ref{proposition:f_and_g}, 
part~\ref{prop:f_one_var_permutation}, 
$ F $ is a permutation of $ \F_p $, 
and by Lemma~\ref{lemma:power_identities} this completes the proof 
of Case~1. 

\medskip

Case 2:  $\ovl{ n_1 } = 1 $.  

As  $ \ovl{ m_1 } \le  \ovl{ n_1 } = 1 $, $ \ovl{ m_1 }  = 1 $. 
We 
claim that the first coordinate 
$ f\mleft( x , y \mright) $ 
of an isomorphism 
$ \varphi \colon D_1 \to D_2 $ 
depends only on $ x $. 
Indeed, two distinct vertices 
$ \mleft( a , b \mright) $ 
and 
$ \mleft( c , d \mright) $ 
in $ D_i $ 
have no common in-neighbor in 
\[
    D_i^- \mleft( a , b \mright)
    =
    \mleft\{ 
        \mleft( 
            s 
            , 
            -b + s^{ m_i } a^{ n_i }
        \mright) 
        : 
        s \in \F_p 
    \mright\}
    \qquad 
    \mt{and} 
    \qquad 
    D_i^-\mleft( c , d \mright)
    =
    \mleft\{ 
        \mleft( 
            t
            , 
            -d + t^{ m_i } c^{ n_i }
        \mright) 
        : 
        t \in \F_p 
    \mright\}
\] 
if and only if the equation 
$ x^{ m_i } \mleft( a^{ n_i } - c^{ n_i } \mright) = b - d $ 
has no solution. 
A solution does not exist if and only if 
$ a^{ n_i } = c^{ n_i } $, 
which is equivalent to 
$ a = c $. 
Thus, in both digraphs, two distinct vertices have disjoint 
in-neighborhoods if and only if their first coordinates agree. 
This property is preserved by 
$ \varphi $, 
and therefore 
$ f\mleft( x , y \mright) = \widehat f \mleft( x \mright) $ 
for some permutation 
$ \widehat f $ 
of 
$ \F_p $, 
with  
$ \widehat f \mleft( 0 \mright) = 0 $, 
$ \widehat f \mleft( 1 \mright) = 1 $. 

We now claim that 
$ g $ 
and 
$ \widehat f $ 
satisfy conditions~\eqref{eqn:g_z_eq_z_f_powers} 
(where $ f \mleft( x , y \mright) $ is replaced by 
$ \widehat f \mleft( x \mright) $). 
Once this is established, the remainder of the proof is identical 
to the argument of Case~1.

Since 
for every $ y \in \F_p $, 
$ 
    \mleft( 1 , y \mright) 
    \to 
    \mleft( a , a^{ n_1 } - y \mright) 
$ 
in 
$ D_1 $, 
we have 
\begin{equation} 
    \label{eqn:g_and_f_second_case}
    g\mleft( y \mright) 
    + 
    g\mleft( a^{ n_1 } - y \mright) 
    =
    \widehat f \mleft( a \mright)^{ n_2 }%
    \td
\end{equation}
For 
$ y = 0 $, 
\eqref{eqn:g_and_f_second_case} 
implies 
$   
    g\mleft( 0 \mright) 
    + 
    g\mleft( a^{ n_1 } \mright) 
    =
    g\mleft( a^{ n_1 } \mright) 
    =
    \widehat f \mleft( a \mright)^{ n_2 } 
$. 
Replacing 
$ \widehat f \mleft( a \mright)^{ n_2 } $ 
by 
$ g\mleft( a^{ n_1 } \mright) $ 
in~\eqref{eqn:g_and_f_second_case}, 
yields 
$ 
    g\mleft( y \mright) 
    + 
    g\mleft( a^{ n_1 } - y \mright) 
    =
    g\mleft( a^{ n_1 } \mright) 
$. 
Let 
$ v = a^{ n_1 } $. 
Then 
$ 
    g\mleft( y \mright) 
    +
    g\mleft( v - y \mright) 
    =
    g\mleft( v \mright) 
$ 
for all 
$ v , y \in \F_p $. 
Setting 
$ v = y + z $ 
yields additivity of 
$ g $ 
on 
$ \F_p $, which together with 
$ g\mleft( 1 \mright) = \widehat f \mleft( 1 \mright)^{ n_2 } = 1 $
implies 
$ g\mleft( z \mright) = z $. 
Finally, applying $ \varphi $ 
to the arcs 
$ \mleft( x , 0 \mright) \to \mleft( 1 , x^{ m_1 } \mright) $ 
and 
$ \mleft( 1 , 0 \mright) \to \mleft( x , x^{ n_1 } \mright) $ 
in 
$ D_1 $ 
gives 
$ \widehat f \mleft( x \mright)^{ m_2 } = x^{ m_1 } $ 
and 
$ \widehat f \mleft( x \mright)^{ n_2 } = x^{ n_1 } $, 
and conditions~\eqref{eqn:g_z_eq_z_f_powers} are established. 
This completes the proof of the theorem in 
Case~2, and so the theorem is proved.
\end{proof}
\section{%
    An Infinite Family of Counterexamples to 
    the Converse of 
    Proposition~\ref{proposition:sufficiency}
}
\label{section:infinite_counterexamples}
An exhaustive computation shows that the smallest 
counterexample to the converse of 
Proposition~\ref{proposition:sufficiency} 
is the pair 
$ D\mleft( 25 ; 3 , 6 \mright) $ 
and 
$ D\mleft( 25 ; 3 , 18 \mright) $, 
both with 
$ 625 $ 
vertices 
and 
$ 15625 $ 
arcs. 
The proof that these two digraphs are isomorphic generalizes 
to an infinite family in the following way. 

\begin{theorem}
\label{thm:(m,n) cong (m,rn)}
Let 
$ R $ 
be an odd prime power such that
$R \equiv 1 \mybmod 4$, 
and put
$ q = R^2 $, 
$ m = \mleft( R + 1 \mright) / 2 $.
Then
$
    D( q ; m , 2 m ) 
    \cong 
    D( q ; m , 2 m^2)
    \cong 
    D( q ; m^2 , 2 m )%
$. 
Moreover, 
there is no integer 
$ k $ 
for which 
$ 
    k
    \mleft ( m , 2 m \mright) 
    \equiv 
    \mleft ( m , 2 m^2 \mright) 
$ 
or 
$ 
    k
    \mleft ( m , 2 m \mright) 
    \equiv 
    \mleft ( m^2 , 2 m \mright) 
$.
\end{theorem}

\begin{proof}
We first prove that 
$ 
    D\mleft( q ; m , 2 m \mright) 
    \cong
    D\mleft( q ; m , 2 m^2 \mright) 
$. 

Put
$ a = (R - 1)/2 $. 
Since 
$ R \equiv 1 \mybmod 4 $, 
$ a $ is even, and 
$ m = \mleft( R + 1 \mright) / 2 $ 
is odd.  

Let $ \alpha $ be a primitive element of $ \F_q $,
and put $ \beta = \alpha^m $.
Since $ q - 1 = 4 a m $, the order of $ \beta $ is
$ 4 a = 2 \mleft( R - 1 \mright) $.
Consequently, $ \beta \notin \F_R $,
while $ \beta^2 $ has order $ 2 a = R - 1 $
and therefore generates $ \F_R^{ \mystar } $.
Thus $ \mleft\{ 1 , \beta \mright\} $
is a basis of $ \F_q $ over $ \F_R $, 
$ \beta^{ 2 i } \in \F_R^{ \mystar } $ 
and 
$  \beta^{ 2 i + 1 } \in \beta\F_R^{ \mystar } $ 
for every 
$ i \in \Z $.
Since 
$ a $ 
is even, 
$ 3 a $ 
is even, and therefore
$ \beta^{ 3 a } \in \F_R^{ \mystar }$. 
Obviously, 
$ \mleft\{ \beta^{ 3 a } , \beta \mright\} $ 
is another basis of 
$ \F_q $ 
over 
$ \F_R $. 
Therefore, the transformation 
\[
    T( \lambda \cdot 1 + \mu \beta )
    =
    \lambda \beta^{ 3 a } + \mu \beta%
    \tc
    \quad
    \lambda , \mu 
    \in 
    \F_R%
\] 
maps one basis to another, and so is an invertible $ \F_R $-linear 
transformation of $ \F_q $. 

Applying $ T $ to powers of 
$ \beta $ we compute 
\begin{equation}
    \label{eqn:T_on_beta_powers_rm_n}
    T( \beta^j )
    =
    \begin{cases}
        \beta^{j+3a}%
        \tc
        & 
        \mt{if $ j $ is even}%
        \tc\\
        \beta^j%
        \tc
        & 
        \mt{if $ j $ is odd}%
        \td
    \end{cases}
\end{equation}

Define a mapping  
$ \sigma $ 
on 
$ \F_q $ 
by
$ \sigma ( 0 ) = 0 $, 
and, for 
$ \alpha^j \in \F_q^{ \mystar } $,
\[
    \sigma( \alpha^j )
    =
    \alpha^{ j + s( j ) }%
    \tc
    \qquad
    \mt{where }\,
    s(j)
    =
    \begin{cases}
        \hphantom{2}a%
        \tc 
        & 
        \mt{if $ j $ is even}%
        \tc\\
        2a%
        \tc 
        & 
        \mt{if $ j $ is odd}%
        \td
    \end{cases}
\]

Since $a$ is even, 
$ s(j) $ is always even,
and the map $ j \mapsto j + s\mleft( j \mright) $ 
preserves parity and acts, in an obvious way, as a translation on 
each parity class of $ \Z_{q-1} $.
As the map is injective, it is a permutation of 
$ \Z_{q-1} $.
Hence, 
$ \sigma $ 
is a bijection of 
$ \F_q$. 

We claim that
\begin{equation}
    \label{eqn:T_and_sigma}
    T( x^m y^{ 2 m } )
    =
    \sigma( x )^m
    \sigma( y )^{ 2 m^2 }
    \quad 
    \mt{ for all  $ x , y \in \F_q $}%
    \td
\end{equation}
It suffices to check this for 
$ x , y \in \F_q^{ \mystar } $,
since both sides are zero if either 
$ x =0 $ 
or 
$ y = 0 $.
Write
$ x = \alpha^u $ 
and 
$ y = \alpha^v $. 
Since 
$ \beta = \alpha^m $, 
we have
$
    x^m y^{ 2 m }
    =
    \alpha^{ m u + 2 m v }
    =
    \beta^{ u + 2 v }%
$. 
Therefore, by~\eqref{eqn:T_on_beta_powers_rm_n}, 
\[
    T\mleft( x^m y^{ 2 m} \mright)
    =
    T\mleft( \beta^{ u + 2 v } \mright)
    =
    \begin{cases}
        \beta^{ u + 2 v + 3 a }%
        \tc 
        &\mt{if $ u $ is even}%
        \tc
        \\
        \beta^{ u + 2 v }%
        \tc
        & 
        \mt{if $ u $ is odd}%
        \td
    \end{cases}
\]

On the other hand,
\[
    \sigma( x )^m
    \sigma( y )^{ 2 m^2 }
    =
    \alpha^{ m( u + s( u ) ) }
    \alpha^{ 2 m^2 ( v + s( v ) ) }
    =
    \beta^{ u + s( u ) + 2 m ( v + s( v ) ) }%
    \td
\]
Thus, to prove 
\eqref{eqn:T_and_sigma}, 
it suffices to check
that for every $ u $ and $ v $, 
\begin{equation}
    \label{eqn:u_s_m_congruence_mod_4a}
    u + s( u ) + 2 m ( v + s( v ) )%
    \equiv
    \begin{cases}
        u + 2 v + 3 a%
        \tc 
        &\mt{if $ u $ is even}%
        \tc
        \\
        u + 2 v%
        \tc
        & 
        \mt{if $ u $ is odd}%
    \end{cases}
    \mymod { 4 a }%
    \td
\end{equation}

Note that since $ m = a + 1 $, 
\[
    2 m ( v + s( v ) )
    =
    \mleft\{
    \begin{aligned}
        ( 2 a + 2 )
        ( v + a )%
        &\equiv 
        2 v + 2 a 
        &&\mymod 
        4 a%
        \tc 
        &&\mt{if $ v $ is even}%
        \tc
        \\
        ( 2 a + 2 ) 
        ( v + 2 a )%
        &\equiv 
        2 v + 2 a 
        &&\mymod 
        4 a%
        \tc 
        &&\mt{if $ v $ is odd}%
        \td
    \end{aligned}
    \mright.
\]

Hence, the left-hand side 
of~\eqref{eqn:u_s_m_congruence_mod_4a}  
is
\[
    u + s( u ) + 2 m ( v + s( v ) )
    \equiv
    u + s( u ) + 2 v + 2 a
    \mymod { 4 a }%
\]
for every $ v $. 
Since 
$ s\mleft( u \mright) =  a $ 
if 
$ u $ 
is even, 
and 
$ s\mleft( u \mright) = 2 a $ 
if 
$ u $ 
is odd, 
we have established congruence~\eqref{eqn:u_s_m_congruence_mod_4a}
and, consequently, 
identity~\eqref{eqn:T_and_sigma}.

Now define 
$ 
    \varphi 
    \colon 
    \F_q^2 
    \to 
    \F_q^2 
$ 
by
$ 
    \varphi ( x , y )
    =
    (
        \sigma( x )
        ,
        T( y )
    )
$. 
If 
$ ( x_1 , y_1 ) \to  ( x_2 , y_2 ) $ 
in 
$ D( q ; m , 2 m ) $, 
then
$ y_1 + y_2 = x_1^m x_2^{ 2 m } $. 
Applying 
$ T $ 
and using~\eqref{eqn:T_and_sigma} 
we obtain 
\[
    T( y_1 ) + T( y_2 )
    =
    T \left( y_1 + y_2 \right) 
    =
    T( x_1^m x_2^{ 2 m } )
    =
    \sigma( x_1 )^m
    \sigma( x_2 )^{ 2 m^2 }%
    \td
\]
Hence, 
$
    \bigl(
        \sigma\mleft( x_1 \mright)
        ,
        T\left( y_1 \right)
    \bigr)
    \to 
    \bigl(
        \sigma\left( x_2 \right)
        ,
        T\left( y_2 \right)
    \bigr)
$
is an arc in 
$ D( q ; m , 2 m^2 ) $.  
Since 
$ \varphi $ 
is bijective and maps arcs to 
arcs, and both digraphs have 
$ q^3 $ arcs, 
$ \varphi $ 
is an isomorphism from 
$
    D( q ; m , 2 m )
$ 
to 
$
    D( q ; m , 2 m^2 )
$.
\medskip

We now prove that 
$
    D( q ; m , 2 m ) 
    \cong 
    D( q ; m^2 , 2 m )%
$.

Put $ k = m + 4 a $. 
Since 
$ \gcd\mleft( m , 4 a \mright) = 1 $, 
we have 
$ 
    \gcd\mleft( k , 4 a \mright) 
    =
    \gcd\mleft( k , m \mright) 
    =
    1
$, 
and consequently $ \ovl{ k } = 1 $. 
Moreover, 
$ k \equiv m \mybmod 4 a $, 
so 
$ 
    k
    \mleft( m , 2 m^2 \mright) 
    \equiv 
    \mleft( m^2 , 2 m^3 \mright) 
    \equiv 
    \mleft( m^2 , 2 m \mright) 
$, 
where the last congruence follows from 
\[
    2 m^3 - 2 m 
    =
    2 m \mleft( m - 1 \mright) \mleft( m + 1 \mright) 
    =
    2 am \mleft( a + 2 \mright) 
    =
    \mleft( q - 1 \mright) 
    \cdot 
    \cbfrac{ a + 2 }{ 2 } 
    \equiv 
    0%
    \td
\]
Proposition~\ref{proposition:sufficiency} now gives 
the result. 

It remains to prove the two nonexistence assertions. 
Every scalar multiple 
$ \mleft( u , v \mright) $ 
of 
$ \mleft( m , 2 m \mright) $ 
has 
$ v \equiv 2 u $. 
However, 
$
    2 m^2 - 2 m 
    =
    2 m \mleft( m - 1 \mright) 
    =
    2 m a 
    =
    \mleft( q - 1 \mright) / 2 
    \not\equiv 
    0%
$, a contradiction proving both assertions. 
\end{proof}

The smallest 
$ q $ 
satisfying the conditions of 
Theorem~\ref{thm:(m,n) cong (m,rn)} 
is 
$ q = 25 $. 
In this case 
$ R = 5 $, 
and hence, 
$ a = ( R - 1 ) / 2 = 2 $,
$ m = 3 $,
$ n = 6 $.
Thus the theorem gives
$ 
    D( 25 ; 3 , 6 ) 
    \cong 
    D( 25 ; 3 , 18 ) 
    \cong 
    D( 25 ; 9 , 6 ) 
$.

\section{Concluding Remarks}
\label{section:concluding_remarks}
Our Theorem~\ref{theorem:prime_classification}, together with
Proposition~\ref{proposition:sufficiency}, gives a simple necessary and
sufficient condition for two monomial digraphs over a prime field 
to be isomorphic. 

For prime powers that are not primes, Theorem~\ref{thm:(m,n) cong (m,rn)} shows
that isomorphic digraphs need not have parameter pairs related by simultaneous
multiplication by an integer relatively prime to $ q - 1 $. The remaining
problem is to determine which parameter pairs yield  isomorphic monomial
digraphs over nonprime finite fields.

Another question is the following. 
Let 
$ D_1 = D\mleft( q ; m_1 , n_1 \mright) $ 
and 
$ D_2 = D\mleft( q ; m_2 , n_2 \mright) $. 
For any digraphs $ H $ and $ D $, let $ n\mleft( H , D \mright) $
denote the number of subdigraphs of $ D $ isomorphic to $ H $.
Does there exist a finite family $ \mathcal H $ of digraphs,
such that, for sufficiently large prime power $ q $ and all parameter pairs,
the equalities 
$
    n\mleft( H , D_1 \mright)
    =
    n\mleft( H , D_2 \mright)
$
for every $ H \in \mathcal H $
imply $ D_1 \cong D_2 $?
The authors have not been able to find such a family. 

A more general question is finding a polynomial time algorithm in $ q $ 
that decides whether $ D_1 $ and $ D_2 $ are isomorphic. 
\printbibliography
\end{document}